\documentclass{amsart}

\usepackage{a4wide,amssymb,amsmath,amsthm,color}

\usepackage[bookmarksopen=true,bookmarksnumbered=true, bookmarksopenlevel=2, colorlinks=true,linkcolor=mblue,
citecolor=mblue,urlcolor=mblue, pdfstartview=FitH]{hyperref}

\usepackage{enumerate}

\definecolor{mblue}{rgb}{0,0,.8}

\newcommand{\N}{\mathbb N}
\newcommand{\Z}{\mathbb Z}
\newcommand{\Q}{\mathbb Q}

\newcommand{\R}{\mathbb R}
\newcommand{\C}{\mathbb C}
\newcommand{\p}{\mathfrak p}

\renewcommand{\a}{\alpha}
\renewcommand{\b}{\beta}

\renewcommand{\c}{\chi}

\renewcommand{\k}{\kappa}
\newcommand{\z}{\zeta}

\newcommand{\B}{\mathcal{B}}

\newcommand\es[0]{e^\ast}
\newcommand\Es[0]{E^\ast}

\newtheorem{thm}{Theorem}[section]
\newtheorem{lem}[thm]{Lemma}
\newtheorem{remark}[thm]{Remark}

\newtheorem{prop}[thm]{Proposition}

\newtheorem{thmx}{Theorem}

 \DeclareMathOperator{\SL}{SL}

\begin{document}

\title[A Conjecture of Coleman on the Eisenstein Family]{A Conjecture of Coleman on the Eisenstein Family}
\author[Bryan W.\ Advocaat, Ian Kiming, Gabor Wiese]{Bryan W.\ Advocaat, Ian Kiming, Gabor Wiese}
\address[Bryan W.\ Advocaat]{Department of Mathematics, University of Luxembourg, Maison du nombre, 6~avenue de la Fonte, L-4364 Esch-sur-Alzette, Luxembourg, and Department of Mathematical Sciences, University of Copenhagen, Universitetsparken 5, DK-2100 Copenhagen \O ,
Denmark.}
\email{bryan.advocaat@uni.lu}
\address[Ian Kiming]{Department of Mathematical Sciences, University of Copenhagen, Universitetsparken 5, DK-2100 Copenhagen \O ,
Denmark.}
\email{kiming@math.ku.dk}
\address[Gabor Wiese]{Department of Mathematics, University of Luxembourg, Maison du nombre, 6~avenue de la Fonte, L-4364 Esch-sur-Alzette,
Luxembourg}
\email{gabor.wiese@uni.lu}


\begin{abstract} We prove for primes $p\ge 5$ a conjecture of Coleman on the analytic continuation of the family of modular functions $\frac{\Es_\k}{V(\Es_\k)}$ derived from the family of Eisenstein series $\Es_\k$.

The precise, quantitative formulation of the conjecture involved a certain on $p$ depending constant. We show by an example that the conjecture with the constant that Coleman conjectured cannot hold in general for all primes. On the other hand, the constant that we give is also shown not to be optimal in all cases.

The conjecture is motivated by its connection to certain central statements in works by Buzzard and Kilford, and by Roe, concerning the ``halo'' conjecture for the primes $2$ and $3$, respectively. We show how our results generalize those statements and comment on possible future developments.
\end{abstract}

\maketitle

\section{Introduction}\label{sec:intro}

In what follows, $p$ will denote a fixed prime $\ge 5$. We let $v_p$ denote the $p$-adic valuation of $\C_p$ normalized so that $v_p(p)=1$.

The conjecture of Coleman referred to in the title is Conjecture 1.1 of Coleman's paper \cite{coleman_eisenstein}. Let us briefly recall the setup as in \cite{coleman_eisenstein}: let $\mathcal{W}$ be the analytic group of continuous $\C_p$-valued characters on $\Z_p^{\times}$ with the subgroup $\B$ consisting of those characters that are trivial on the $(p-1)$st roots of unity. For $\kappa\in \B \backslash \{ 1\}$ we have the family $E^{\ast}_{\kappa}$ of Eisenstein series with $q$-expansions
$$
\Es_{\kappa}(q) = 1 + \frac{2}{\zeta^{\ast}(\kappa)} \sum_{n=1}^{\infty} \left( \sum_{\substack{ d\mid n \\ p\nmid d}} \kappa(d) d^{-1} \right) \cdot q^n
$$
with $\zeta^{\ast}$ the $p$-adic zeta function on $\mathcal{W}$.

Our convention in the present paper is that $\N = \Z_{\ge 1}$. If $k\in (p-1) \N$ then $k$ defines an element $\k\in \B \backslash \{ 1\}$ by $x\mapsto x^k$. For such $k$ we shall abuse notation and identify the corresponding $\kappa$ with $k$. The specialization to $\kappa =k$ of the Eisenstein family gives us the classical Eisenstein series $E_k^{\ast}$ with $q$-expansion
$$
\Es_k(q) = 1 + \frac{2}{(1-p^{k-1}) \zeta(1-k)} \cdot \sum_{n=1}^{\infty} \left( \sum_{\substack{d\mid n \\ p\nmid d}} d^{k-1} \right) q^n
$$
(with $\zeta$ the Riemann zeta function) that is an Eisenstein series of weight $k$ on $\Gamma_0(p)$.

Furthermore, we denote by $E_k$ the standard, normalized Eisenstein series of level~$1$ and even integer weight~$k\ge 4$.

Recall that we have a function $w$ on $\mathcal{W}$ defined by $w(\kappa) := \kappa(1+p) - 1$. Thus, for $k\in (p-1)\N$ we have $w(k) = (1+p)^k - 1$.

The setting for Coleman's conjecture is as follows. Suppose that $\k \in \B \backslash \{ 1\}$. Let $V$ be the $p$-adic Frobenius operator, acting on $q$-expansions as $q\mapsto q^p$. Coleman had already shown that the $p$-adic modular function $\frac{\Es_{\kappa}}{V(\Es_{\kappa})}$ is defined on the ordinary locus of $X := X_1(p)$ and defines an overconvergent function, cf.\ p.\ 2946 of \cite{coleman_eisenstein}, or the reference on that page, or, alternatively, \cite[Corollary 2.1.1]{coleman_char_series_U}. (This is a function that can also be considered when $p=2,3$.) Conjecture 1.1 of \cite{coleman_eisenstein} is a precise prediction of how far into the supersingular region this function converges, i.e., what is its rate of overconvergence. In this formulation, the conjecture also represents a conjectural answer to a question posed in Coleman and Mazur's foundational paper \cite{coleman_mazur_eigencurve} on the eigencurve, -- see the remarks at the end of p.\ 43 of \cite{coleman_mazur_eigencurve}.

The following theorem proves a version of the conjecture. When we say ``version'', what we primarily mean is that the constant $c_p$ appearing in the theorem is not precisely the constant that Coleman was expecting in his conjecture (for primes $p\ge 5$.) We shall comment further upon that below, but would like here to note that we do not believe that Coleman's conjecture is true with the exact value of the constant that he gave (that would correspond to being able to take $c_p=1$ in our theorem.) We shall discuss this in detail below, and especially in section \ref{sec:remarks}.

To formulate our main theorem, we find the following notation convenient: $f\in M_0(O,\ge \rho)$ means that $f$ is an overconvergent function of tame level $1$, defined over $O$ that is $r$-overconvergent whenever $v_p(r) < \rho$. See below in section \ref{sec:prelim} for a few additional details on this notation.

\begin{thmx}\label{thm:main_C} There is a constant $0< c_p < 1$ such that the following holds. Let $\k \in \B \backslash \{ 1\}$ be a character and let $O$ be the ring of integers in the extension of $\Q_p$ generated by the values of $\k$.

Then
$$
\frac{\Es_{\k}}{V(\Es_{\k})} \in M_0(O,\ge c_p \cdot \min \{ 1, v_p(w(\k)) \} ) .
$$

Explicitly, we can take
$$
c_p = \frac{2}{3} \cdot \left( 1 - \frac{p}{(p-1)^2} \right) \cdot \frac{1}{p+1} .
$$
\end{thmx}

We see the background and motivation for Coleman's conjecture as follows. A conjecture about the behavior of $U$ near the boundary of weight space, the ``halo'' conjecture, \cite[Conjecture 2.5]{wan_xiao_zhang_eigencurve}, \cite[Conjecture 1.2]{liu_wan_xiao_eigencurve}, \cite[Conjecture 1.2]{bergdall_pollack_fredholm}, seems to be attributed to Coleman, but has also developed from the main result of the paper \cite{buzzard_kilford_2-adic} by Buzzard and Kilford that can now be seen as establishing that conjecture for the prime $p=2$. Subsequently, Roe established the conjecture in \cite{roe_3-adic} for $p=3$ by similar methods. What Coleman did in \cite{coleman_eisenstein} was first to reinterpret certain central, indeed decisive, results from \cite{buzzard_kilford_2-adic}, \cite{roe_3-adic}, specifically \cite[Theorem 7]{buzzard_kilford_2-adic}, \cite[Theorem 4.2]{roe_3-adic}, as a precise statement, for $p=2,3$, about rates of overconvergence of the functions $\frac{\Es_{\k}}{V(\Es_{\k})}$. Coleman then went on and conjectured in \cite[Conjecture 1.1]{coleman_eisenstein}) a similar and precise statement for all primes $p\ge 5$.

It will be seen that Theorem \ref{thm:main_C} is not asserting precisely the same as what \cite[Conjecture 1.1]{coleman_eisenstein} expects for primes $p\ge 5$. First, \cite[Conjecture 1.1]{coleman_eisenstein} is formulated from a rigid analytic perspective. Though this is unimportant as far as the substance of the statement is concerned, we shall comment briefly on it at the beginning of section \ref{subsec:coleman_conj_comparison} below. Secondly, and more importantly, the precise value that we give for the constant $c_p$ (the reader should note that Coleman's $c_p$ denotes something else than our $c_p$) is not precisely what \cite[Conjecture 1.1]{coleman_eisenstein} would expect for primes $p\ge 5$: though formulated in rigid analytic terms, we can reinterpret \cite[Conjecture 1.1]{coleman_eisenstein} as expecting the statement of Theorem \ref{thm:main_C}, but with the value $c_p=1$ for the constant in the theorem. Below in section \ref{subsec:coleman_conj_comparison} we will show by means of a numerical example that we cannot take $c_p=1$ in Theorem \ref{thm:main_C}. Thus, the precise formulation of \cite[Conjecture 1.1]{coleman_eisenstein} for primes $p\ge 5$ appears to us to have been a too optimistic extrapolation from the cases $p=2,3$.

On the other hand, we also do not claim optimality of the constant $c_p$ in our Theorem \ref{thm:main_C}, at least not for all primes. Thus, in section \ref{subsec:p=5,7}, using certain ad hoc arguments, we will show that $c_p$ can be improved a little bit for the cases $p=5,7$.

We will derive Theorem \ref{thm:main_C} from Theorem \ref{thm:main_B} below that may be of some interest in itself. It gives a ``formal Katz expansion'' and a lower bound for the valuation of its coefficients.
Since in the theorem as well as in the proof we will be talking about Katz expansions (\cite[Section 2.6]{katz_padic} or \cite[Section 4.1]{vonk_oc}) of overconvergent $p$-adic modular functions, we briefly remind the reader of these: as Katz first showed, there is for each $i \in \N$ a direct sum decomposition
$$
M_{i(p-1)}(\Z_p) = E_{p-1} \cdot M_{(i-1)(p-1)}(\Z_p) \oplus B_i(\Z_p) .
$$
of $\Z_p$-modules where $M_k$ denotes modular forms of weight $k$ on $\SL_2(\Z)$ (for the proof of this, one can refer to Katz' original work, \cite[Proposition 2.8.1]{katz_padic}, but a simple, elementary proof is also possible by using ``Victor Miller'' bases in level $1$, see e.g.\ \cite[Section 5]{kr_eisenstein}.)

The splitting is not unique, but we will fix a specific choice for the $B_i$ in section \ref{sec:proof_main_thm} below. Katz expansions of the modular functions we will be working with then take the form $\sum_{i=0}^{\infty} \frac{b_i}{E_{p-1}^i}$ with $b_i\in B_i(\Z_p)$ (we put $B_0(\Z_p) = \Z_p$.)

\begin{thmx}\label{thm:main_B} (a) There are modular forms $b_{ij}\in B_i(\Z_p)$ for each $i,j \in \Z_{\ge 0}$ such that the following holds. If $\kappa\in \B \backslash \{ 1\}$ then the Katz expansion of the modular function $\frac{\Es_\k}{V(\Es_\k)}$ is
$$
\frac{\Es_\k}{V(\Es_\k)} = \sum_{i=0}^{\infty} \frac{\b_i(w(\k))}{E_{p-1}^{i}}
$$
where
$$
\b_i(w(\k)) := \sum_{j=0}^{\infty} b_{ij} w(\k)^j
$$
for each $i$.

\noindent (b) There is a constant $c_p$ with $0< c_p <1$ such that for the modular forms $b_{ij}$ in part (a) we have
$$
v_p(b_{ij}) \ge c_p i - j
$$
for all $i,j$.

In fact, we can take the explicit constant $c_p$ from Theorem \ref{thm:main_C}.
\end{thmx}

The plan of the paper is as follows. After setting up notation and various preliminaries in the next section, in section \ref{sec:proof_main_thm} we first prove part (a) of Theorem \ref{thm:main_B}. We derive that part conveniently as an application of the existence of ``Victor Miller'' bases for modular forms in level $1$.

The idea of proof of the more difficult part (b) of Theorem \ref{thm:main_B} is to utilize the fact that the paper \cite{kr_eisenstein} gives us information about rates of overconvergence of $p$-adic modular functions of form $\frac{E^{\ast}_k}{V(E^{\ast}_k)}$ with $k\in (p-1)\N$. The observation that these infinitely many ``data points'' imply the divisibility properties of part (b) is the technical core of the paper, and it depends on the combinatorial/linear algebra Proposition \ref{prop:val_poly_coeff_b}. Given that proposition, the proofs of part (b) of Theorem \ref{thm:main_B} and after that of Theorem \ref{thm:main_C} proceed along straightforward lines.

Finally, in section \ref{sec:remarks} we comment on Coleman's original conjecture as compared with our Theorem \ref{thm:main_C} as well as on the question of optimality of the constant $c_p$. We also show that our results can be used to generalize certain statements from the papers \cite{buzzard_kilford_2-adic, roe_3-adic} pertaining to the study of the $U$ operator in weights $\k$ with $0<v_p(w(\k))<1$.

We need the condition $p\ge 5$ primarily for the usual reasons such as that $E_{p-1}$ is a modular form, but occasionally in more general discussion and remarks the condition can be relaxed. We will indicate when that is the case.

We close the paper with some remarks about the context of this work.
Our paper follows Coleman's original approach (\cite{coleman_banach}) to the existence of what we now refer to as Coleman families of modular forms, which builds on the family $\frac{E^{\ast}_k}{V(E^{\ast}_k)}$ of $p$-adic modular functions.
Whereas this approach has now been superseded by a more intrinsic, geometric definition of the notion of an overconvergent modular form of arbitrary weight, cf.\ the works of Pilloni \cite{pilloni}, and Andreatta, Iovita, Stevens \cite{andreatta_iovita_stevens}, we feel that Coleman's original way is still very valuable, in particular, due to its explicit nature, which we exploit and explore in this article.
Especially, if one wants to do explicit, computational work, for instance with Coleman families, of which there are very few, explicit examples, the approach using Eisenstein series (see \cite{coleman_stevens_teitelbaum} and \cite{destefano} for examples with small primes) might still have merit and in fact might at this point in time be the only option. We hope to return to this point in future work.

Finally, we would like to mention that the paper \cite{ye} is concerned, as are we, with problems of extending modular forms further into the supersingular locus. At this point, though, neither do we see immediate implications of that paper for the problems we are addressing here, nor vice versa.

\subsection*{Acknowledgements}
This work was supported by the Luxembourg National Research Fund PRIDE17/1224660/GPS.

\section{Notation and preliminaries}\label{sec:prelim} Consider a finite extension $K$ of $\Q_p$ with ring of integers $O$. By $M_k(O)$ we shall denote the $O$-module of weight $k$ modular forms on $\SL_2(\Z)$ with coefficients in $O$.

For $r\in O$ we can talk about $r$-overconvergent modular forms of tame level $1$. We will only be dealing with weight $0$ forms, i.e., modular functions, that are holomorphic at $\infty$. The $r$-overconvergent of these with coefficients in $O$ form an $O$-module that we will denote by $M_0(O,r)$.

Most of our arguments will proceed via consideration of ``the'' Katz expansion of such forms: for each $i \in \N$, there is a (non-unique) direct sum decomposition
$$
M_{i(p-1)}(\Z_p) = E_{p-1} \cdot M_{(i-1)(p-1)}(\Z_p) \oplus B_i(\Z_p) .
$$

In section \ref{sec:proof_main_thm} we will make a specific, fixed choice of these splittings that is convenient both theoretically and computationally. For now it suffices to say that an element $f\in M_0(O,r)$ has a ``Katz expansion''
$$
f = \sum_{i=0}^{\infty} \frac{b_i}{E_{p-1}^i}
$$
where $b_i\in B_i(O)$ satisfy $v_p(b_i) \ge i\cdot v_p(r)$ for all $i$, as well as $v_p(b_i) - i\cdot v_p(r) \rightarrow \infty$ for $i\rightarrow \infty$. This expansion is unique once the splittings above have been fixed. One should note that these expansions are not necessarily exactly the ones that Katz introduced in \cite{katz_padic} (the reason being that he used a the geometric language and had to contend with the usual issues when the level is $1$.) However, all we will be concerned with are growth properties of the valuations of the $b_i$ and these are independent of which splitting we use. But see the more general discussion in \cite[Section 2]{kr_eisenstein}, for instance. We should note here that the modules $B_i$ obviously depend on $p$ though out of convenience we will suppress that information from our notation.

From \cite{kr_eisenstein} we shall also borrow the following notation. For a rational $\rho \in [0,1]$ let $M_0(O,\ge \rho)$ denote the $O$-module of forms $f$ such that $f\in M_0(O,r)$ for some $r$, and such that for the coefficients $b_i$ of the Katz expansion of $f$ we have $v_p(b_i) \ge \rho \cdot i$ for all $i$.

Elementary considerations (\cite[Proposition 2.3]{kr_eisenstein}) show that an element $f\in M_0(O,r)$ is in $M_0(O,\ge \rho)$ if and only if $f\in M_k(O',r')$ whenever $K'/K$ is a finite extension with ring of integers $O'$ and $r'\in O'$ satisfies $v_p(r') < \rho$. Again, this is the case if and only if we have $f\in M_0(O,r)$ for some $r$ as well as $f\in M_k(O',r')$ for a sequence of finite extensions $K'/K$ with rings of integers $O'$ and elements $r'$ such that $v_p(r')$ converges to $\rho$ from below.

\section{Proof of the main theorems}\label{sec:proof_main_thm}

\subsection{Existence of the ``formal Katz expansion''}\label{subsec:proof_thmBa}

\begin{proof}[Proof of Theorem \ref{thm:main_B}, part (a)] We start the proof by repeating the observation made in section 5 of \cite{buzzard_kilford_2-adic} that with $w = w(\k)$ we have a formal power series expansion
$$
\frac{\Es_\k}{V(\Es_\k)} = \sum_{n=0}^{\infty} \left( \sum_{j=0}^{\infty} a_{nj} w^j \right) \cdot q^n \in \Z_p[[w,q]]
$$
in the sense that if we specialize $w$ on the right hand side to $w=w(\k)$ for a character $\kappa \in \B \backslash \{ 1\}$, we obtain the $q$-expansion of the function $\frac{\Es_\k}{V(\Es_\k)}$. (The argument at the beginning of \cite[Section 5]{buzzard_kilford_2-adic} is for $p=2$, but carries over to a general prime $p$.)

For Katz expansions at tame level $1$ it is both theoretically and computationally convenient to use the idea of Lauder \cite{lauder_comp} of exploiting the existence of ``Miller bases'' for modular forms of level $1$: Put $d_{s(p-1)} := \dim M_{s(p-1)}(\Q_p)$. There are splittings
$$
M_{i(p-1)}(\Z_p) = E_{p-1} \cdot M_{(i-1)(p-1)}(\Z_p) \oplus B_i(\Z_p) .
$$
of $\Z_p$-modules where the free $\Z_p$-module $B_i(\Z_p)$ has a basis
$$
\{ g_{i,j} \mid d_{(i-1)(p-1)} \le j \le d_{i(p-1)} - 1 \}
$$
with the property that the $q$-expansion of $g_{i,j}$ starts with $q^j$ (for $i=0$ the definition is $g_{0,0} := 1$.) Cf.\ for instance \cite[Section 5]{kr_eisenstein} for explicit formulas for the $g_{i,j}$. This means that the (infinite) matrix that has the coefficients of the $q$-expansions
$$
g_{0,0}(q), \ldots, g_{i,d_{(i-1)(p-1)}}(q),\ldots ,g_{i,d_{i(p-1)}-1}(q),\ldots,
$$
as rows will be upper triangular with $1$s in the diagonal.

Since $E_{p-1}(q) \in 1 + pq \Z_p[[q]]$ we also have, formally, $E_{p-1}^{-i}(q) \in 1 + pq \Z_p[[q]]$ for each $i\ge 0$. Thus, the matrix with rows the coefficients of the (formal) $q$-expansions
$$
g_{0,0}(q) \cdot 1, \ldots, g_{i,d_{(i-1)(p-1)}}(q) E_{p-1}^{-i}(q),\ldots ,g_{i,d_{i(p-1)}-1}(q) E_{p-1}^{-i}(q),\ldots,
$$
is again upper triangular with $1$s in the diagonal.

It follows from these considerations that we have an isomorphism $\phi: \prod_{i\ge 0} B_i(\Z_p) \cong \Z_p[[q]]$ of $\Z_p$-modules given by
$$
\phi((b_i)_{i\ge 0}) := \sum_{i\ge 0} b_i(q) E_{p-1}^{-i}(q) .
$$

In particular, for each $j$ we have a sequence of unique elements $b_{ij}\in B_i(\Z_p)$, $i\ge 0$, such that
$$
\sum_{n=0}^{\infty} a_{nj} q^n = \sum_{i=0}^{\infty} \frac{b_{ij}(q)}{E_{p-1}^i(q)} .
$$

Define then
$$
H(w,q) := \sum_{i=0}^{\infty} \frac{\sum_{j=0}^{\infty} b_{ij}(q) w^j}{E_{p-1}^i(q)}
$$
as a formal power series in $w$ and $q$ with coefficients in $\Z_p$.

Consider now a character $\kappa \in \B \backslash \{ 1\}$. Let $O$ be the ring of integers of an extension of $\Q_p$ large enough to contain the values of $\k$. Let $\p$ be the maximal ideal of $O$ and let us consider the specialization $H(w(\k),q)$ modulo $\p^m$ for a fixed $m\in\N$. As $v_p(w(\k))>0$ there is $j(m)\in \N$ such that $w(\k)^j \equiv 0 \pmod{\p^m}$ for $j > j(m)$. We then find in $O/\p^m$:
\begin{eqnarray*} H(w(\k),q) & = & \sum_{i=0}^{\infty} \frac{\sum_{j=0}^{\infty} b_{ij}(q) w(\k)^j}{E_{p-1}^i(q)} \equiv \sum_{i=0}^{\infty} \frac{\sum_{j=0}^{j(m)} b_{ij}(q) w(\k)^j}{E_{p-1}^i(q)} = \sum_{j=0}^{j(m)} \left( \sum_{i=0}^{\infty} \frac{b_{ij}(q)}{E_{p-1}^i(q)} \right) w(\k)^j \\
& = & \sum_{j=0}^{j(m)} \left( \sum_{n=0}^{\infty} a_{nj} q^n \right) w(\k)^j = \sum_{n=0}^{\infty} \left( \sum_{j=0}^{j(m)} a_{nj} w(\k)^j \right) q^n \\
& \equiv & \sum_{n=0}^{\infty} \left( \sum_{j=0}^{\infty} a_{nj} w(\k)^j \right) q^n = \frac{\Es_\k}{V(\Es_\k)}(q) .
\end{eqnarray*}
As this congruence holds for all $m\in\N$ we conclude that
$$
\frac{\Es_\k}{V(\Es_\k)}(q) = H(w(\k),q) = \sum_{i=0}^{\infty} \frac{\sum_{j=0}^{\infty} b_{ij}(q) w(\k)^j}{E_{p-1}^i(q)}
$$
in $O[[q]]$.

Now, as we remarked above the function $\frac{\Es_\k}{V(\Es_\k)}$ is an overconvergent modular function with a Katz expansion
$$
\frac{\Es_\k}{V(\Es_\k)} = \sum_{i=0}^{\infty} \frac{\b_i(\k)}{E_{p-1}^{i}}
$$
where $\b_i(\k)\in B_i(O)$ for all $i$. Then $\sum_{i=0}^{\infty} \frac{\b_i(\k)(q)}{E_{p-1}^{i}(q)} = \sum_{i=0}^{\infty} \frac{\sum_{j=0}^{\infty} b_{ij}(q) w(\k)^j}{E_{p-1}^i(q)}$ in $O[[q]]$ and so
$\b_i(k)(q) = \sum_{j=0}^{\infty} b_{ij}(q) w(\k)^j$ for all $i$ by the injectivity of the isomorphism $\phi$ above.
Then $\b_i(\k) = \sum_{j=0}^{\infty} b_{ij} w(\k)^j$ for all $i$ by the $q$-expansion principle.
\end{proof}

As explained in the introduction above, the non-trivial part of Theorem \ref{thm:main_B} is part (b) that will be obtained by using information from \cite{kr_eisenstein}, specifically information about the overconvergence of modular functions $\frac{E^{\ast}_k}{V(E^{\ast}_k)}$ for classical weights $k\in (p-1)\N$: if we combine information about the rate of overconvergence of these modular functions, cf.\ \cite[Theorem A]{kr_eisenstein}, with part (a) of Theorem \ref{thm:main_B}, we obtain a statement about the growth w.r.t.\ $i$ of the valuations of infinite sums
$$
\sum_{j=0}^{\infty} b_{ij} w^j
$$
with $w$ corresponding to such classical weights. The combinatorial and linear algebra observations of the next subsection will show that this suffices to make a statement about the valuations of the modular forms $b_{ij}$ themselves.

\subsection{Valuations of the inverse Vandermonde matrix}\label{subsec:comb_linalg}

In this section, $p$ is any prime number.
Let $n \in \N$ and $x_0,\dots,x_{n-1} \in \C_p$ be pairwise distinct.
Consider the Vandermonde matrix
\[ V = V(x_0,\dots,x_{n-1}) = \left(\begin{smallmatrix}
1&x_0&x_0^2& \cdots & x_0^{n-1} \\
1&x_1&x_1^2& \cdots & x_1^{n-1} \\
1&x_2&x_2^2& \cdots & x_2^{n-1} \\
\vdots & \vdots & \vdots & \ddots & \vdots\\
1&x_{n-1}&x_{n-1}^2& \cdots & x_{n-1}^{n-1} \\
\end{smallmatrix}\right).\]

The following lemma appears to be well-known, but we provide the short proof for lack of a convenient reference.

\begin{lem}\label{lem:invV}
Let $0 \le i,j \le n-1$. Then the coefficient of the matrix $V(x_0,\dots,x_{n-1})^{-1}$ at position $(i+1,j+1)$ equals
\[ (-1)^{n-1-i} \cdot \frac{s_{n-1-i}(x_0,\dots,\hat{x_j},\dots,x_{n-1})}{\prod_{0 \le \ell \le n-1, \ell \neq j} (x_j - x_\ell)},\]
where $s_d(\dots)$ is the elementary symmetric polynomial of degree $0 \le d \le n-1$ in $n-1$ variables (the hat in $\hat{x_j}$ means that the variable $x_j$ is omitted).
\end{lem}

\begin{proof}
We start from the formula defining the elementary symmetric polynomials in $n-1$ variables
$\prod_{\ell=1}^{n-1} (T-t_\ell) = \sum_{i=0}^{n-1} (-1)^{n-1-i} \cdot s_{n-1-i}(t_1,\dots,t_{n-1}) \cdot T^i$
and replace $(t_1,\dots,t_{n-1})$ by $(x_0,\dots,\hat{x_j},\dots,x_{n-1})$ and $T$ by $x_k$ for $0 \le j,k \le n-1$, leading to
$\prod_{0 \le \ell \le n-1, \ell \neq j} (x_k - x_\ell)  = \sum_{i=0}^{n-1}  (-1)^{n-1-i} \cdot s_{n-1-i}(x_0,\dots,\hat{x_j},\dots,x_{n-1})\cdot x_k^i$, implying the claim.
\end{proof}

We need to study the valuations of denominators occurring in the inverse Vandermonde matrix. In the next proposition we prove a bit more than we will actually need for the proof of Theorem \ref{thm:main_B}. We do this in order to show that the estimates that we get from the proposition are in fact optimal.

\begin{prop}\label{prop:val_denomV} Let $S\subseteq \Z_p^{\times}$ be a finite subset and put $n:=|S|$. For $x\in S$ put
$$
v(S,x) := v_p\left(\prod_{s \in S, s \neq x} (x-s) \right) = \sum_{s \in S, s \neq x} v_p(x-s) .
$$

Then
$$
\max_{x \in S} v(S,x) \ge \sum_{i=1}^\infty \left\lfloor \frac{n-1}{(p-1)p^{i-1}} \right\rfloor =: f(n) .
$$

Furthermore, for each $n\in \N$ there exists $S\subseteq \Z_p^{\times}$ with $|S|=n$ such that $\max_{x \in S} v(S,x) = f(n)$. For instance, one has equality if $S$ consists of the first $n$ natural numbers prime to $p$.
\end{prop}

\begin{proof} For the beginning of the argument we allow $S$ more generally to be any non-empty finite subset of $\Z_p$ of cardinality $n$.

Let $r(S)$ be the maximal $r \in \Z_{\ge 0}$ such that all elements in $S$ are congruent to each other modulo~$p^r$. We write $S' := \{ s \,\mathrm{div}\, p^{r(S)} \;|\; s \in S \}$ where $s \,\mathrm{div}\, p^r$ denotes the number $\sum_{i\ge r} a_i p^{i-r} $ if $s=\sum_{i\ge 0} a_i p^i$ is the standard $p$-adic expansion of $s$, i.e., with the $a_i$ in $\{ 0,\ldots,p-1\}$. Thus, if $s$ is an ordinary integer, $s \,\mathrm{div}\, p^r$ is the quotient of division with remainder of $s$ by $p^r$. We observe that $|S| = |S'|$.

For any $d \in \Z/p\Z$, let $S_d = \{s \,\mathrm{div}\, p \;|\; s \in S, s \equiv d \pmod{p}\}$.

By the definition of $S'$ and for any $d \in  \Z/p\Z$, we have
\begin{align*}
\max_{x \in S} v(S,x)
&=   r(S) \cdot (|S|-1) + \max_{x \in S'} v(S',x) \\
&\ge r(S) \cdot (|S|-1) + \max_{x \in S', ~x \equiv d \pmod{p}} v(S',x)\\
&=   r(S) \cdot (|S|-1) + ( |(S')_d| -1 ) +  \max_{x \in (S')_d} v((S')_d,x)
\end{align*}
because only those $s \in S'$ contribute to $\sum_{s \in S', s \neq x} v_p(x-s)$ that are congruent to $x$ modulo~$p$.

Now, for cardinality reasons there must exist $d \in  \Z/p\Z$ such that $|(S')_d| \ge \left\lceil \frac{|S'|}{p} \right\rceil = \left\lceil \frac{|S|}{p} \right\rceil$.
Applying this we obtain
$$
\max_{x \in S} v(S,x) \ge r(S) \cdot (|S|-1) + \left\lceil \frac{|S|}{p} \right\rceil - 1 + \max_{x \in (S')_d} v((S')_d,x)
$$
from which we can see the inequality
\begin{equation}\label{eq:SZp}
\max_{x \in S} v(S,x) \ge r(S) \cdot (|S|-1) + \sum_{i=1}^\infty\left(\left\lceil \frac{|S|}{p^i} \right\rceil -1\right)
\end{equation}
by induction on $|S|$: if $|S|=1$ the statement is trivial. If $|S|>1$ then for the induction step we use that $|(S')_d|<|S'|=|S|$ for any $d$, apply the induction hypothesis to $\max_{x \in (S')_d} v((S')_d,x)$, drop the term $r((S')_d) \cdot (|(S')_d|-1)$, and use the inequality $|(S')_d| \ge \left\lceil \frac{|S|}{p} \right\rceil$.

The inequality \eqref{eq:SZp} obviously implies the inequality
\begin{equation}\label{eq:SZp'}
\max_{x \in S} v(S,x) \ge \sum_{i=1}^\infty\left(\left\lceil \frac{|S|}{p^i} \right\rceil -1\right) = \sum_{i=1}^\infty \left\lfloor \frac{|S|-1}{p^i} \right\rfloor .
\end{equation}

Let us now assume the setup of the proposition, i.e., that $S \subseteq \Z_p^\times$. Assume first that $r(S)=0$ so that $S'=S$. In that case, we can improve \eqref{eq:SZp} slightly because there now exists $d \in  (\Z/p\Z)^\times$ such that $|(S')_d|=|S_d| \ge \left\lceil \frac{|S|}{p-1} \right\rceil$. Then as above we have
\[ \max_{x \in S} v(S,x) \ge  ( |S_d| -1 ) +  \max_{x \in S_d} v(S_d,x) \ge \left\lceil \frac{|S|}{p-1} \right\rceil -1 +  \max_{x \in S_d} v(S_d,x).\]
We now apply \eqref{eq:SZp'} to the right most term and obtain
\begin{equation}\label{eq:SZptimes}
\max_{x \in S} v(S,x) \ge \sum_{i=1}^\infty \left(\left\lceil \frac{|S|}{(p-1)p^{i-1}} \right\rceil -1\right) = \sum_{i=1}^\infty \left\lfloor \frac{|S|-1}{(p-1)p^{i-1}} \right\rfloor = f(n) .
\end{equation}

Next we claim that this formula also holds when $r(S) \ge 1$. Indeed, applying again \eqref{eq:SZp}, we have
\begin{align*}
\max_{x \in S} v(S,x)
&\ge r(S) \cdot (|S|-1) + \sum_{i=1}^\infty\left(\left\lceil \frac{|S|}{p^i} \right\rceil -1\right) = r(S)(|S|-1) + \sum_{i=1}^\infty\left(\left\lfloor \frac{|S|-1}{p^i} \right\rfloor\right)\\
&\ge \sum_{i=0}^\infty\left(\left\lfloor \frac{|S|-1}{p^i} \right\rfloor\right)\ge   \sum_{i=0}^\infty\left(\left\lfloor \frac{1}{p-1} \cdot \frac{(|S|-1)}{p^i} \right\rfloor\right)
= \sum_{i=1}^\infty \left\lfloor \frac{|S|-1}{(p-1)p^{i-1}} \right\rfloor .\\
\end{align*}

Moreover, the above analysis shows that \eqref{eq:SZptimes} is an equality if there is a sequence $d_0,d_1,\ldots \in \Z/p\Z$ (only finitely many terms matter) such that the recursively defined sets $S^{(0)}:= S$ and $S^{(i+1)} = (S^{(i)})_{d_i}$ for $i \ge 0$ satisfy $r(S^{(i)})=0$ for all~$i \ge 0$, as well as $|S^{(1)}| = \left\lceil \frac{|S|}{p-1} \right\rceil$ and $|S^{(i+1)}| = \left\lceil \frac{|S^{(i)}|}{p} \right\rceil$ for all $i\ge 1$.

If $S$ consists of the first $n$ natural numbers prime to $p$ we take $d_0=1$ and $d_i =0$ for $i\ge 1$ and then these conditions are actually satisfied: writing $n=(p-1)q+r$ with $q\ge 0$, $0\le r < p-1$ one verifies that $S_1$ consists of the first $\lceil \frac{n}{p-1} \rceil = q + \lceil \frac{r}{p-1} \rceil$ consecutive integers starting from $0$. One also sees that if $\N \ni m\ge 0$ and $\Sigma = \{ 0,\ldots, m-1\}$ then $\Sigma_0$ consists of the first $\lceil \frac{m}{p} \rceil$ consecutive integers, starting from $0$.
\end{proof}

\begin{prop}\label{prop:val_poly_coeff_b}
Let $n\in \N$ and let $x_0,\dots,x_{n-1} \in \Z_p^\times$ be any set of units such that
$$
\max_{0 \le j \le n-1} v_p\left(\prod_{0 \le \ell \le n-1, \ell \neq j} (x_j - x_\ell) \right) = f(n) := \sum_{i=1}^\infty \left\lfloor \frac{n-1}{(p-1)p^{i-1}} \right\rfloor \le (n-1)\cdot \frac{p}{(p-1)^2} ,
$$
the existence of which is assured by Proposition \ref{prop:val_denomV}.
\begin{enumerate}[(a)]
\item
Let $V=V(x_0,\dots,x_{n-1})$ be the Vandermonde matrix. Then the $p$-valuation of all coefficients of $V^{-1}$ is at least $-f(n)\ge (1-n)\cdot \frac{p}{(p-1)^2}$.

\item
If $m\in \R$ and if $b_0,\dots,b_{n-1} \in \C_p$ satisfy
$$
v_p(b_0 + b_1 x_i + \cdots b_{n-1} x_i^{n-1}) \ge m
$$
for all $0 \le i \le n-1$, then
$$
v_p(b_j) \ge m-f(n) \ge m-(n-1)\cdot \frac{p}{(p-1)^2}
$$
for all $0 \le j \le n-1$.
In particular, if $m=n$ we will have
$$
v_p(b_j) \ge \left( 1 - \frac{p}{(p-1)^2} \right) \cdot n
$$
for each $j$.
\end{enumerate}
\end{prop}

\begin{proof}
For the first inequality, observe that $f(n) \le \frac{n-1}{p-1} \cdot \sum_{i=0}^{\infty} \frac{1}{p^i} = (n-1)\cdot \frac{p}{(p-1)^2}$.
Part~(a) is a direct consequence of Lemma~\ref{lem:invV} and the choice of $x_0,\dots,x_{n-1} \in \Z_p^\times$.
Part~(b) follows by considering $b_0 + b_1 x_i + \cdots b_{n-1} x_i^{n-1}$ as $V$ times the vector of the $b_i$.
\end{proof}

\begin{remark}\label{rem:optimality_upper_bound} As we will see immediately below, the main ingredient from this section in the proof of Theorem \ref{thm:main_B} is Proposition \ref{prop:val_poly_coeff_b}. As we also see, the essential statement of Proposition \ref{prop:val_poly_coeff_b} is that we can choose units $x_0,\dots,x_{n-1} \in \Z_p^\times$ in such a way that we have a good upper bound for $\max_{0 \le j \le n-1} v_p\left(\prod_{0 \le \ell \le n-1, \ell \neq j} (x_j - x_\ell) \right)$. A choice of the units $x_i$ is provided by the second part of Proposition \ref{prop:val_denomV}. The purpose of the first part of Proposition \ref{prop:val_denomV} is to show that this upper bound is optimal.
\end{remark}

\subsection{Proof of Theorem \ref{thm:main_B}, part (b)}\label{subsec:proof_thmBb}
Considering the modular forms $b_{ij}\in B_i(\Z_p)$ from part (a), we will show that $v_p(b_{ij}) \ge c_p \cdot i - j$ for all $i,j$ where $c_p := \frac{2}{3} \cdot \left( 1 - \frac{p}{(p-1)^2} \right) \cdot \frac{1}{p+1}$.

Fix $i_0 \in \Z_{\ge 0}$ and let us show that $v_p(b_{i_0 j} p^j) \ge c_p \cdot i_0$ for all $j\ge 0$. As $b_{i_0 j}$ has coefficients in $\Z_p$ we certainly have $v_p(b_{i_0 j} p^j) \ge j$, and so the claim is clear if $j \ge n$ with
$$
n := \left\lceil \frac{2}{3} \cdot \frac{1}{p+1} \cdot i_0 \right\rceil
$$
as we will then have $v_p(b_{i_0 j} p^j) \ge n > c_p \cdot i_0$. Thus, we must show
$$
v_p(b_{i_0 j} p^j) \ge c_p \cdot i_0
$$
for $j=0,\ldots,n-1$.

Consider classical weights $k\in \N$ divisible by $p-1$. For such weights the corresponding point $w(k)$ in weight space is $w(k) := (1+p)^k - 1$. By part (a) of Theorem~\ref{thm:main_B}, the $i_0$th coefficient in the Katz expansion of the $p$-adic modular function $\frac{E^{\ast}_k}{V(E^{\ast}_k)}$ is
$$
\sum_{j=0}^{\infty} b_{i_0 j} w(k)^j .
$$

The crucial ingredient in the proof is now the observation that we know from \cite[Theorem A]{kr_eisenstein} that
$$
v_p(\sum_{j=0}^{\infty} b_{i_0 j} w(k)^j) \ge \frac{2}{3} \cdot \frac{1}{p+1} \cdot i_0
$$
(more precisely: \cite[Remark 4.2]{kr_eisenstein} combined with the proof of \cite[Theorem A]{kr_eisenstein} shows that we have $\frac{E^{\ast}_k}{V(E^{\ast}_k)} \in M_0(\Z_p, \ge \frac{2}{3} \cdot \frac{1}{p+1})$.)

Now, recalling again that $b_{i_0 j}$ has coefficients in $\Z_p$ and combining this with the fact that $w(k) \in p\Z$ for classical weights $k\equiv 0 \pmod{p-1}$ as above, we find from the definition of $n := \lceil \frac{2}{3} \cdot \frac{1}{p+1} \cdot i_0 \rceil$ that
$$
v_p( \sum_{j=0}^{n-1} b_{i_0 j} w(k)^j) \ge \left\lceil \frac{2}{3} \cdot \frac{1}{p+1} \cdot i_0 \right\rceil = n
$$
for every such classical weight $k$.

We write the sum on the left hand side as $\sum_{j=0}^{n-1} (b_{i_0 j} p^j) \cdot \left( \frac{w(k)}{p} \right)^j$ and notice that elementary considerations show that the numbers
$$
\frac{w(k)}{p} = \frac{(1+p)^k - 1}{p}
$$
are dense in $\Z_p$ when $k$ ranges over the classical weights $\equiv 0 \pmod{p-1}$. Then we see that Proposition~\ref{prop:val_poly_coeff_b} can be applied to deduce that $v_p(b_{i_0 j} p^j) \ge \left( 1 - \frac{p}{(p-1)^2} \right) \cdot n$ for $j=0,\ldots,n-1$. Indeed, a lower bound $v_p(b_{i_0 j}p^j) \ge m$ is equivalent to having the same lower bound for the valuations of all Fourier coefficients of $b_{i_0 j}p^j$. But then we have
$$
v_p(b_{i_0 j}p^j) \ge \left( 1 - \frac{p}{(p-1)^2} \right) \cdot n \ge \left( 1 - \frac{p}{(p-1)^2} \right) \cdot \frac{2}{3} \cdot \frac{1}{p+1} \cdot i_0 = c_p \cdot i_0
$$
for $j=0,\ldots,n-1$, and we are done.

\subsection{Proof of Theorem \ref{thm:main_C}}\label{subsec:proof_thmC} Put $w_0 := w(\kappa)$. By part (a) of Theorem \ref{thm:main_B} we have a Katz expansion
$$
\frac{\Es_{\k}}{V(\Es_{\k})} = \sum_{i=0}^{\infty} \frac{\b_i(w_0)}{E_{p-1}^{i}} .
$$
with $\b_i(w_0) := \sum_{j=0}^{\infty} b_{ij} w_0^j$. Referring back to the remarks of section \ref{sec:prelim}, all we have to show is that we have $v_p(\sum_{j=0}^{\infty} b_{ij} w_0^j) \ge c_p \cdot \min \{ 1, v_p(w_0) \} \cdot i$ for all $i$. To see this, fix an $i$, write $\rho := c_p i$, and split the sum as
$$
\sum_{j=0}^{\infty} b_{ij} w_0^j = \sum_{0\le j \le \rho} b_{ij} w_0^j  +  \sum_{j>\rho} b_{ij} w_0^j .
$$

For the terms in the first sum note that part (b) of Theorem \ref{thm:main_B} implies that their valuations are bounded from below by
$$
c_pi - j + j v_p(w_0) = c_p i - j (1-v_p(w_0)) .
$$

If now $v_p(w_0) \le 1$ this is $\ge c_p i - \rho (1-v_p(w_0)) = \rho v_p(w_0) = c_p v_p(w_0) \cdot i$, and if $v_p(w_0) \ge 1$, this is certainly $\ge c_p \cdot i$.

On the other hand, as $b_{ij} \in B_i(\Z_p)$ for all $j$, the terms in the second sum have valuations bounded from below by $j v_p(w_0) > \rho v_p(w_0) = c_p v_p(w_0) \cdot i$. We are done.

\section{Further remarks and results}\label{sec:remarks}

\subsection{The original conjecture of Coleman}\label{subsec:coleman_conj_comparison}

\subsubsection{} Coleman's conjecture \cite[Conjecture 1.1]{coleman_eisenstein} is formulated in rigid analytic terms as a conjecture concerning analytic continuation of $\frac{\Es_{\k}}{V(\Es_{\k})}$ considered as a function of two variables $(P,\k) \in X_1(p) \times \B$, $\k \neq 1$. As a consequence of Coleman's earlier results on the nonvanishing of $\frac{\Es_{\k}}{V(\Es_{\k})}$ on $Z$ (cf.\ the remarks at the bottom of p.\ 2946 of \cite{coleman_eisenstein}), this function is initially defined for $P$ in the ordinary locus $Z$ where $E_{p-1}(P)$ is a unit. Given our Theorem \ref{thm:main_B}, the value of the function at such a point is the value of the converging infinite sum
$$
\sum_{i=0}^{\infty} \left( \sum_{j=0}^{\infty} b_{ij}(P) w^j \right) E_{p-1}(P)^{-i} .
$$
where we have written $w := w(\k)$. The question is how far into the supersingular region this function extends when $v_p(w(\k)) < 1$. Let us give the core argument showing that the function extends under the condition $\frac{1}{c_p} v_p(E_{p-1}(P)) < v_p(w) < 1$ (for primes $p\ge 5$, \cite[Conjecture 1.1]{coleman_eisenstein} would say this, but with $c_p=1$.) We give the argument assuming that $P$ corresponds to an elliptic curve defined over the ring $O_0$ of integers in a finite extension of $\Q_p$. Let us choose an extension $K$ of $\Q_p$ large enough to contain $O_0$ and $w$ as well as an element $\a$ with $v_p(\a) = c_p$. Let $O$ denote the ring of integers of $K$. We can then see that the above series converges to an element of $O$: Rewriting the series as
$$
\sum_{i=0}^{\infty} \left( \sum_{j=0}^{\infty} b_{ij}(P) w^{j-c_p i} \right) \left( E_{p-1}(P)^{-1} w^{c_p} \right)^i ,
$$
since $v_p( E_{p-1}(P)^{-1} w^{c_p} ) > 0$, we see that it suffices to show that $\sum_{j=0}^{\infty} b_{ij}(P) w^{j-c_p i}$ for any fixed $i\ge 0$ converges to an element of $O$. To do so, fix an $i\ge 0$ and split up this sum as
$$
\sum_{0\le j \le c_p i} b_{ij}(P) w^{j-c_p i} + \sum_{j > c_p i} b_{ij}(P) w^{j-c_p i} .
$$

In the second sum we have $w^{j-c_p i} \in O$ for each term, and since $v_p(w)>0$, the sum converges.

The first sum is finite, and so convergence is not an issue, but we still need to see that the sum gives an element of $O$. But if for $j\le c_p i$ we define the modular form $\tilde{b}_{ij}$ to be $\tilde{b}_{ij} := \a^{-i} p^j b_{ij}$ then Theorem \ref{thm:main_B} (and the $q$-expansion principle) implies that $\tilde{b}_{ij}$ is a modular form defined over $O$ so that the value $\tilde{b}_{ij}(P)$ is in $O$. Now,
$$
b_{ij}(P) w^{j-c_p i} = \tilde{b}_{ij}(P) \cdot \a^i p^{-j} w^{j-c_p i},
$$
and since
$$
v_p(\a^i p^{-j} w^{j-c_p i}) = c_p i - j + (j-c_p i) v_p(w) \ge 0
$$
as $j\le c_p i$ and $v_p(w) < 1$, we are done.

\subsubsection{} We now show by a numerical example one cannot take $c_p=1$ in Theorem \ref{thm:main_C}: let $p=5$ and let $\c$ be the Dirichlet character of conductor $5^2$ given by $\c(7) = 1$, $\c(6) = \z$ with $\z$ a primitive $5$th root of unit. Then $\c$ can be viewed as a character on $\Z_5^{\times}$ and as such is trivial on the $4$th roots of unity. Let $\k$ be the character on $\Z_5^{\times}$ given by $\k(x) = x^4 \c(x)$. Then $\Es_{\k}$ is a classical Eisenstein series of weight $4$ on $\Gamma_1(5^2)$ with nebentypus $\c$. We have $v_5(w(\k)) = \frac{1}{4}$, and so, if we could take $c_5=1$ in Theorem \ref{thm:main_C} we would be able to conclude (via Theorem \ref{thm:main_C}) that $\Es_{\k}/V(\Es_{\k}) \in M_0(O, \ge \frac{1}{4})$ with $O$ the ring of integers of $\Q_5(\z)$. But a computation shows this not to be the case: recall that for $p=2,3,5,7,13$ where $X_0(p)$ has genus $0$, the function
$$
f_p(z) := \left( \frac{\eta(pz)}{\eta(z)}\right)^{\frac{24}{p-1}}
$$
with $\eta$ the Dedekind eta-function is a {\it Hauptmodul} for $\Gamma_0(p)$, {\it i.e.}, a generator of the function field of $X_0(p)$. D.\ Loeffler has shown, cf.\ \cite[Corollary 2.2]{loeffler}, that if $c$ is a constant with $v_p(c) = \frac{12}{p-1} v_p(r)$ then the powers of $c f_p$ give an orthonormal basis for the $r$-overconvergent modular functions of tame level $1$. Hence, if we consider the expansion
$$
\frac{\Es_{\k}}{V(\Es_{\k})} = \sum_{i=0}^{\infty} a_i f_5^i ,
$$
then the statement that $\Es_{\k}/V(\Es_{\k}) \in M_0(O, \ge \frac{1}{4})$ together with \cite[Corollary 2.2]{loeffler} implies $v_5(a_i) \ge \frac{3}{4} \cdot i$ for all $i$. But, the expansion is easy to compute from $q$-expansions as the $q$-expansion of $f_5$ starts with $q$, and one finds that $v_5(a_{10}) = 1$.

\subsubsection{} It appears to us that the precise, quantitative form of \cite[Conjecture 1.1]{coleman_eisenstein} for primes $p\ge 5$ resulted from an optimistic extrapolation from the cases $p=2,3$. Coleman proved \cite[Conjecture 1.1]{coleman_eisenstein} for $p=2,3$ as a consequence of \cite[Theorem 7]{buzzard_kilford_2-adic} and \cite[Theorem 4.2]{roe_3-adic}, respectively, theorems that are quite central in those papers. The primes $2$ and $3$ differ from primes $p\ge 5$ for all the usual reasons, but in this specific setting there are additional differences: as an inspection of the proofs of \cite[Theorem 7]{buzzard_kilford_2-adic} and \cite[Theorem 4.2]{roe_3-adic} shows, the fact that the $U$ operator at tame level $1$ and for these primes enjoys particularly strong integrality properties plays a significant role in the proofs. Those stronger integrality properties fail for primes $p\ge 5$, which is also one reason why the arguments in these papers do not generalize for primes $p\ge 5$ in any straightforward manner, as far as we can see. The stronger integrality properties of $U$ for $p=2,3$ can ultimately be seen to derive from the fact that the exponent $\frac{24}{p-1}$ occurring in the definition of the {\it Hauptmodul} $f_p$ above is divisible by $p$ precisely when $p\in \{2,3\}$.

One further observation on the difference between the cases $p=2,3$ and $p\ge 5$ is as follows. If one considers the shape of the statements of \cite[Theorem 7]{buzzard_kilford_2-adic} and \cite[Theorem 4.2]{roe_3-adic}, a naive generalization to primes $p\ge 5$ would be a statement of form $v_p(b_{ij}) \ge d_p (i-j)$ in part (b) of Theorem \ref{thm:main_B}, with some constant $d_p$ depending on $p$. Extensive numerical calculation of the $v_p(b_{ij})$, the details of which will be reported on elsewhere, strongly suggests that such a statement does not hold, but that the correct lower bound for primes $p\ge 5$ is in fact a statement of the form in part (b) of Theorem \ref{thm:main_B}. Again we see this difference between the cases $p=2,3$ and $p\ge 5$ as being connected with the above stronger integrality properties of $U$.

\subsection{The constant \texorpdfstring{$c_p$}{cp}}\label{subsec:p=5,7} We will now discuss the specific constant $c_p$ that appears in Theorems \ref{thm:main_C} and \ref{thm:main_B}. In particular, we will show that it is not optimal, at least not for all primes. We show this by improving the constant in the cases $p=5,7$ by certain ad hoc arguments, specifically:

\begin{prop}\label{prop:special_const_p=5,7} For $p=5,7$ we can take $c_p=\left( 1 - \frac{p}{(p-1)^2} \right) \cdot \frac{p-1}{p(p+1)} = \frac{p^2-3p+1}{p(p^2-1)}$ in Theorems \ref{thm:main_C} and \ref{thm:main_B}.
\end{prop}

Notice first from the proof of part (b) of Theorem \ref{thm:main_B} that the constant appears as the product of two factors: $c_p = a_p \cdot b_p$ where $a_p := 1 - \frac{p}{(p-1)^2}$ is the constant appearing in Proposition \ref{prop:val_poly_coeff_b} whereas $b_p := \frac{2}{3} \cdot \frac{1}{p+1}$ comes from results of \cite{kr_eisenstein} that imply $\frac{E^{\ast}_k}{V(E^{\ast}_k)} \in M_0(\Z_p, \ge b_p)$ for classical weights $k\in\N$ divisible by $p-1$.

Here, the constant $a_p$ does not seem to admit any essential improvement, cf.\ Remark \ref{rem:optimality_upper_bound}. On the other hand, the constant $b_p$ in the above is not optimal, at least not for all primes. Let us briefly recall the origin of the constant $b_p$ in \cite{kr_eisenstein}: the statement that we have $\frac{E^{\ast}_k}{V(E^{\ast}_k)} \in M_0(\Z_p, \ge b_p)$ for classical weights $k$ divisible by $p-1$ derives from the more precise statement that $\frac{V(E^{\ast}_k)}{E^{\ast}_k} \in \frac{1}{p}M_0(\Z_p, \ge \frac{1}{p+1})$ (\cite[Theorem A]{kr_eisenstein}); judging from numerical experiments, this latter statement actually does appear close to optimal. As arguments in \cite{kr_eisenstein} show, the statement that we have $\frac{E^{\ast}_k}{V(E^{\ast}_k)} \in M_0(\Z_p, \ge b_p)$ with the above value of $b_p$ is obtained as a consequence of the more precise statement coupled with the congruence $E_{n(p-1)} \equiv E_{p-1}^n \pmod{p^2}$ (for primes $p\ge 5$, $n\in \N$.)

For the primes $p=5,7$ we can improve the constant $b_p$ as follows.

\begin{prop}\label{prop:special_oc_p=5,7} If $p\in \{5,7\}$ and $k\in\N$ is divisible by $p-1$ then
$$
\frac{\Es_k}{V(\Es_k)} \in M_0\left( \Z_p, \ge \frac{p-1}{p(p+1)} \right) .
$$
\end{prop}

The proof of Proposition \ref{prop:special_const_p=5,7} now consists of repeating the proof of part (b) of Theorem \ref{thm:main_B} by using Proposition \ref{prop:special_oc_p=5,7} as input.

The proof of Proposition \ref{prop:special_oc_p=5,7} runs along the same general lines of reasoning as were employed in \cite{kr_eisenstein}, see for instance the proof of \cite[Theorem B]{kr_eisenstein}.

The essential point is a consideration of the rate of overconvergence of the $p$-adic modular functions $\es_n := \frac{\Es_{n(p-1)}}{E_{p-1}^n}$ for $n\in \N$. For these we have the following that we will also formulate for the functions $e_n := \frac{E_{n(p-1)}}{E_{p-1}^n}$ as the proof is the same. By a $1$-unit in a ring $M_0(O,r)$ we mean an element of form $1+af$ where $f\in M_0(O,r)$ and $a \in O$ is a constant with $v_p(a)>0$. A $1$-unit is thus invertible in the ring $M_0(O,r)$.

\begin{prop}\label{prop:en*_p=5,7} Let $p\in \{5,7\}$. For $n\in\N$ we have
$$
e_n, \es_n \in M_0\left( \Z_p, \ge \frac{p-1}{p+1} \right) .
$$

As a consequence, $e_n,\es_n$ are $1$-units in $M_0(O,r)$ whenever $O$ is the ring of integers of any sufficiently large, finite extension $K/\Q_p$, and $r\in O$ satisfies $v_p(r) < \frac{p-1}{p+1}$.
\end{prop}

\begin{proof} The argument is the same for $e_n$ and $\es_n$, so let us just consider $\es_n$. For the first statement, considering the Katz expansion
$$
\es_n = 1 + \sum_{i=1}^{\infty} \frac{b_i}{E_{p-1}^i}
$$
of $\es_n$ where $b_i\in B_i(\Z_p)$ and the $B_i(\Z_p)$ as above, we must show that
$$
v_p(b_i) \ge \frac{p-1}{p+1} \cdot i
$$
for all $i$.

Since $\es_n \in \frac{1}{p} \cdot M_0(\Z_p, \ge \frac{p}{p+1} )$ by \cite[Theorem C]{kr_eisenstein} we have $v_p(b_i) \ge -1 + \frac{p}{p+1} \cdot i$ for all $i$. Thus, the desired inequality is seen to hold for $i\ge p+1$ as we then have $\frac{p-1}{p+1} \cdot i \le -1 + \frac{p}{p+1} \cdot i$.

Secondly, by \cite[Lemma 3.11]{kr_eisenstein} we have the congruence $\es_n \equiv 1 \pmod{p^2}$ of $q$-expansions, and by \cite[Proposition 2.5]{kr_eisenstein} this implies $v_p(b_i) \ge 2$ for all $i$. This again implies the desired when $p=5$ and $i=1,2,3$, and when $p=7$ and $i=1,2$.

To deal with the remaining cases, as we noted above in section \ref{subsec:proof_thmBa}, for any $p\ge 5$ the rank of the $\Z_p$-module $B_i(\Z_p)$ equals $d_{i(p-1)} - d_{(i-1)(p-1)}$ where $d_k$ denotes the dimension of the space of modular forms of weight $k$ on $\SL_2(\Z)$.

Consider then $p=5$. We then have $b_4=b_5=0$ because $d_{12}=d_{16}=d_{20}=2$ so that $B_4=B_5=0$. Thus, the desired inequality also holds for $i=4,5$ and hence for all $i$.

Consider then $p=7$. In this case we have $B_3=B_5=B_7=0$ because $d_{12}=d_{18}=2$, $d_{24}=d_{30}=3$, and $d_{36}=d_{42}=4$. Hence $b_3=b_5=b_7=0$, and we only need to verify the inequality for $i=4,6$. But we have $v_7(b_4) \ge \frac{7}{8} \cdot 4 - 1 = \frac{5}{2}$, and since $b_4\in \Z_7$ this implies $v_7(b_4)\ge 3 = \frac{6}{8} \cdot 4$. Similarly, $v_7(b_6) \ge \frac{7}{8} \cdot 6 - 1 = \frac{17}{4}$ whence $v_7(b_6) \ge 5 > \frac{6}{8} \cdot 6$. Thus, the desired inequality holds also for $i=4,6$ and so for all $i$.

Suppose now that $K/\Q_p$ is a finite extension, that $O$ is the ring of integers of $K$, and that $r\in O$ has $v_p(r) < \frac{p-1}{p+1}$. Assume that $K$ is large enough that there exists $a\in O$ with $0 < v_p(a) \le \frac{1}{2} \cdot (\frac{p-1}{p+1} - v_p(r))$. Defining $b_i' := a^{-1} b_i$ for $i\ge 1$ with the $b_i$ from the Katz expansion of $\es_n$ above, we then find
$$
v_p(b_i') - i v_p(r) \ge (i-\frac{1}{2})\cdot (\frac{p-1}{p+1} - v_p(r))
$$
which shows that $v_p(b_i') - i v_p(r) \ge 0$ for $i\ge 1$ as well as $v_p(b_i') - i v_p(r) \rightarrow \infty$ for $i\rightarrow \infty$. But then
$$
f := \sum_{i=1}^{\infty} \frac{b_i'}{E_{p-1}^i}
$$
defines an element of $M_0(O,r)$, and as $\es_n = 1 + a\cdot f$ with $v_p(a)>0$ we see that $\es_n$ is a $1$-unit in $M_0(O,r)$.
\end{proof}

\begin{proof}[Proof of Proposition \ref{prop:special_oc_p=5,7}] Let $p$ be $5$ or $7$, let $k\in\N$ be divisible by $p-1$, and put $n:=k/(p-1)$.

Suppose that $K/\Q_p$ is a finite extension, that $O$ is the ring of integers of $K$, and that $r\in O$ is such that $v_p(r^p) < \frac{p-1}{p+1}$. Suppose further that $K$ is large enough so that the second part of Proposition \ref{prop:en*_p=5,7} applies, i.e., so that $\es_n$ is a $1$-unit in $M_0(O,r^p)$. As the Frobenius operator maps $M_0(O,r^p)$ to $M_0(O,r)$, we can conclude that
$$
V(\es_n) = \frac{V(\Es_k)}{V(E_{p-1})^n}
$$
is a $1$-unit in $M_0(O,r)$. Now, as $v_p(r) < \frac{1}{p+1}$, the ``Coleman--Wan theorem'', \cite[Lemma 2.1]{wan}, tells us that the function $\frac{E_{p-1}}{V(E_{p-1})}$ is a $1$-unit in $M_0(O,r)$. In particular, we have then that
$$
\frac{\Es_k}{V(\Es_k)} = \frac{\es_n}{V(\es_n)} \cdot \left( \frac{E_{p-1}}{V(E_{p-1})} \right)^n \in M_0(O,r) .
$$

As we can choose a sequence of extensions $K/\Q_p$ such that the valuations $v_p(r)$ of the elements $r$ converge to $\frac{p-1}{p(p+1)}$ from below, the proposition follows rom the remarks at the end of section \ref{sec:prelim}.
\end{proof}

Numerical experimentation suggests that Proposition \ref{prop:en*_p=5,7} continues to hold for some primes $p>7$, perhaps for all, though we do not have an explanation at this point.

\subsection{The action of \texorpdfstring{$U$}{U} in weight \texorpdfstring{$\k$}{kappa}} The family $\frac{\Es_\k}{V(\Es_\k)}$ of functions occurs prominently in Coleman's seminal work \cite{coleman_banach} as a tool that enables one to relay the study of the $U$ operator in general weights back to weight $0$. For this to work, some information about the analytical properties of the family is necessary. In the papers \cite{buzzard_kilford_2-adic} and \cite{roe_3-adic} concerning the primes $2$ and $3$, respectively, very detailed information about the family was obtained and used to prove the ``halo'' conjecture in those cases. We will show here that our results permit us to generalize a certain aspect of the analysis of these papers. It would be possible to formulate this more generally for arbitrary primes $p\ge 5$, but for simplicity we will restrict ourselves to ``genus zero primes'', i.e., where $X_0(p)$ has genus zero.

These primes are $p=2,3,5,7,13$. For these primes, instead of the formal Katz expansion of $\frac{\Es_\k}{V(\Es_\k)}$ of Theorem \ref{thm:main_B} one can consider a formal expansion
$$
\frac{\Es_\k}{V(\Es_\k)} = \sum_{i,j\ge 0}^{\infty} a_{ij} w^j t^i
$$
where $w = w(\k)$, $\k \in \B \backslash \{ 1\}$, and where for $t$ we can take $t=f_p = \left( \frac{\eta(pz)}{\eta(z)} \right)^{\frac{24}{p-1}}$ the standard {\it Hauptmodul}, or, alternatively, for $p=2,3$ we can follow the papers \cite{buzzard_kilford_2-adic, roe_3-adic} and take for $t$ a certain uniformizer of $X_0(4)$ (when $p=2$) or $X_0(9)$ (when $p=3$.) In all cases, we will have the coefficients $a_{ij}$ in $\Z_p$ and the expansion has the advantage of being easy to compute for a given $\k$ because the $q$-expansion of $t$ will begin with $q$.

This formal expansion is a central object of study of the papers \cite{buzzard_kilford_2-adic, roe_3-adic} because it gives us information about the action of the $U$ operator on weight $\k$ overconvergent modular forms: for $0\le r < \frac{p}{p+1}$, by choosing $c\in O_{\C_p}$ to be of a specific, on $r$ dependent, absolute value, one has $V(\Es_{\k}) (ct)^i$, $i=0,1,2,\ldots $ as an orthonormal basis for the Banach space of $r$-overconvergent modular forms of weight $\k$. If we choose $t=f_p$, then according to \cite[Corollary 2.2]{loeffler} we should choose $c$ with $v_p(c) = \frac{12r}{p-1}$; for the other choices of $t$, see for instance the discussion on pp.\ 614--615 of \cite{buzzard_kilford_2-adic}. The action of $U$ on this basis can be described via the above expansion of $\frac{\Es_\k}{V(\Es_\k)}$: if we write $U(ct)^i = \sum_j m_{ij} (ct)^j$ (which is of course independent of $\k$) then the (infinite) matrix  giving the action of $U$ on the basis element $V(\Es_{\k}) (ct)^i$ is given by the product
$$
\frac{\Es_\k}{V(\Es_\k)} \cdot \sum_j m_{ij} \cdot V(\Es_{\k})(ct)^j
$$
as is seen by applying the identity $U(V(F)G) = FU(G)$ (``Coleman's trick''.)

A crucial part of the papers \cite{buzzard_kilford_2-adic} ($p=2$) and \cite{roe_3-adic} ($p=3$) consists in showing that when the factor $\frac{\Es_\k}{V(\Es_\k)}$ of the above matrix is properly ``rescaled'' as a function of $t$ then modulo the maximal ideal of $O_{\C_p}$ it becomes independent of $\k$ when $v_p(w(\k))$ is in a certain interval. Let us explain this in detail. Suppose that we have established a lower bound of the form
$$
v_p(a_{ij}) \ge \a i - \b j
$$
with certain positive constants $\a$ and $\b$. Write $w = w(\k)$ and define the power series $g_{\k}(x)$ such that
$$
\frac{\Es_\k}{V(\Es_\k)} = g_{\k}(w^{\gamma} t)
$$
where
$$
\gamma := \frac{\a}{\b} .
$$

We can then see that the coefficients of $g_{\k}$ are integral and the reduction $\bar{g}_{\k}$ of $g_{\k}$ modulo the maximal ideal of $O_{\C_p}$ is independent of $\k$ when $0<v_p(w(\k)) < \b$: writing $g_{\k}(x) = \sum_{n=0}^{\infty} c_n t^n$ we have
$$
c_n = \sum_j a_{nj} w^{j-\gamma n} .
$$

Assume then $0< v_p(w) < \b$. We can then see that each term $a_{nj} w^{j-\gamma n}$ is integral and in the maximal ideal when $j \neq \gamma n$: for $j\ge \gamma n$ this is clear as the $a_{nj}$ are integral. Suppose then that $j < \gamma n$. Then, using $\gamma \b = \a$, we have
$$
v_p(a_{nj} w^{j-n\gamma}) \ge \a n - \b j + (j-n\gamma) v_p(w) = (\gamma n -j)(\b - v_p(w)) > 0 .
$$

In the paper \cite{buzzard_kilford_2-adic} where $p=2$ the above lower bound for the valuations of the $a_{ij}$ was proved with $\a = \b = 3$, cf.\ \cite[Theorem 7]{buzzard_kilford_2-adic} (note that their $a_{ij}$ would be our $a_{ji}$.) Thus $\gamma = 1$, and they were able to conclude that $\bar{g}_{\k}$ is independent of $\k$ when $w=w(\k)$ satisfies $0< v_p(w) < 3$ as well as $\bar{c}_n = \bar{a}_{n,n}$. Similarly, for $p=3$ the paper \cite{roe_3-adic} established a lower bound with $\a = \b = 1$ with analogous conclusions for $\bar{g}_{\k}$.

For the primes $p=5,7,13$ we choose $t=f_p$ in the above, and then arguments completely similar to those that proved part (b) of Theorem \ref{thm:main_B} (working with expansions in $t$ rather than formal Katz expansions) will show that one has
$$
v_p(a_{ij}) \ge d_p i - j
$$
for all $i,j$ where
$$
d_p = \frac{12}{p-1} \cdot c_p
$$
with $c_p$ from Theorem \ref{thm:main_B}, or alternatively for $p=5,7$ from Proposition \ref{prop:special_const_p=5,7}. Here, the factor $\frac{12}{p-1}$ once again comes from \cite[Corollary 2.2]{loeffler}. We can then conclude that we have $\frac{\Es_\k}{V(\Es_\k)} = g_{\k}(w^{d_p} t)$ for a power series $g_{\k}$ with integral coefficients whose reduction $\bar{g}_{\k}$ is independent of $\k$ when $0<v_p(w(\k))<1$.

This statement is of course quite uninteresting unless the constant $d_p$ is optimal as otherwise the reduction $\bar{g}_{\k}$ will just be the constant $1$. However, numerical calculations, at this point mostly for $p=5$, strongly suggests the possibility of identifying the optimal constant and perhaps even the non-trivial reduction $\bar{g}_{\k}$. This will be reported on in detail elsewhere.

\end{document}